\documentclass[12pt]{article}%

\usepackage{amsmath,enumerate}
\usepackage{amsfonts}
\usepackage{amssymb}

\newtheorem{theorem}{Theorem}[section]

\newtheorem{corollary}[theorem]{Corollary}

\newtheorem{lemma}[theorem]{Lemma}

\newenvironment{proof}[1][Proof]{\noindent\textbf{#1.} }
{\hfill \ \rule{0.5em}{0.5em}}

\begin{document}

\title{\vspace{-1cm}The Anti-Ramsey Problem for the 
Sidon equation}
\author{
Vladislav Taranchuk\thanks{Department of Mathematics and Statistics, California State University Sacramento.}
\and
Craig Timmons\thanks{Department of Mathematics and Statistics, California State University Sacramento.  
This work was supported by a grant from the Simons Foundation (\#359419, Craig Timmons).} 
}

\maketitle

\vspace{-1em}

\begin{abstract}
For $n \geq k \geq 4$, let $AR_{X + Y = Z + T}^k (n)$ be the maximum number 
of rainbow solutions to the Sidon equation $X+Y = Z + T$ over all $k$-colorings 
$c:[n] \rightarrow [k]$.  It can be shown that the 
total number of solutions in $[n]$ to the Sidon equation is 
$n^3/12 + O(n^2)$ and so, trivially, $AR_{X+Y = Z + T}^k (n) \leq n^3 /12 + O (n^2)$.
We improve this upper bound to 
\[
AR_{X+Y = Z+ T}^k (n) \leq \left( \frac{1}{12} - \frac{1}{24k} \right)n^3 + O_k(n^2)
\]
for all $n \geq k \geq 4$.  Furthermore, we give an explicit $k$-coloring of $[n]$
with more rainbow solutions to the Sidon equation than a random $k$-coloring, and 
gives a lower bound of 
\[
\left( \frac{1}{12} - \frac{1}{3k} \right)n^3 - O_k (n^2) 
\leq AR_{X+Y = Z+ T}^k (n).
\]
When $k = 4$, we use a different approach based on additive energy to 
obtain an upper bound of $3n^3 / 96 + O(n^2)$, whereas our lower bound 
is $2n^3 / 96 - O (n^2)$ in this case.
\end{abstract}


\section{Introduction}

Most of the notation we use is standard.  For a positive integer $n$, let  
$[n]  = \{1,2, \dots , n \}$.  If $X$ is a set and $m \geq 0$ is an integer, 
then $\binom{X}{m}$ is the set of all subsets of $X$ of size $m$.  
A \emph{$k$-coloring} of a set $X$ is a function $c: X \rightarrow [k]$.  The function 
$c$ need not be onto.  A subset $Y \subset X$ is \emph{monochromatic}
under $c$ if $c(y) = c(y')$ for all $y,y' \in Y$.  The set $Y $ is 
\emph{rainbow} if no two elements of $Y$ have been assigned 
the same color.  

The hypergraph Ramsey Theorem states that for any positive integers $s$, $k$, and $m$, 
there is an $N = N(s,k,m)$ such that for all $n \geq N$ the following holds:  
if $c$ is any $k$-coloring of $\binom{ [n] }{m}$,
then there is a set $S \subset [n]$ such that 
$\binom{ S}{m}$ is monochromatic under $c$.    
This theorem is one of the most important theorems in combinatorics.
Today, Ramsey Theory is a cornerstone in combinatorics 
and there is a vast amount of literature on Ramsey type problems.  
Here we will focus on a Ramsey problem in the integers and recommend
Landman and Robertson \cite{lr} for a more comprehensive introduction to this area.    
The problem we consider is inspired by the investigations of two recent papers.  

In \cite{saad}, Saad and Wolf introduced an arithmetic analog of some problems in 
graph Ramsey Theory.
In particular, given a graph $H$, let $RM_k ( H , n )$ be the 
minimum number of monochromatic copies of $H$ over 
all $k$-colorings $c: E( K_n) \rightarrow [k]$.  The parameter 
$RM_k(H,n)$ is the \emph{Ramsey multiplicity} of $H$ and has been studied for different graphs $H$.    
The arithmetic analog from \cite{saad} replaces graphs with 
linear equations, and sets up a general framework where
these ideas from graph theory (Sidorenko property, Ramsey multiplicty) have
natural counterparts.  Fix an abelian group $\Gamma$.    
If $L$ is a linear equation with integer coefficients, one can look at the minimum 
number of monochromatic solutions to $L$ over all 
$k$-colorings $c: \Gamma \rightarrow [k]$.          
One of the first examples given in \cite{saad} (see Example 1.1) concerns the  
Sidon equation $X + Y = Z+T$.  This famous equation has a rich 
history in combinatorics.  A \emph{Sidon set} in an abelian 
group $\Gamma$ is a set having only trivial solutions to $X+ Y = Z + T$.  
If $A \subset \Gamma$ is a Sidon set and $\Gamma$ is finite, then a simple counting 
argument gives $|A| \leq 2 | \Gamma |^{1/2}   + 1$.  The constant 
2 can be improved in many cases, but 
what concerns us here is that when $A$ is much larger, 
say $|A| = \alpha | \Gamma |$ for some $\alpha >0$, then 
$A$ will certainly contain nontrivial solutions to the Sidon equation. 
Thus, a natural question is given
 a $k$-coloring $c: \Gamma \rightarrow [k]$, at least how 
 many solutions to the Sidon equation must be monochromatic.
 This question, and several others including 
 results on $X+Y = Z$ (Schur triples) and $X+Y = 2Z$ (3 term a.p.'s), is answered by the results of \cite{saad}.
 For more in this direction, we refer the reader to that paper.  
 
Recently, De Silva, Si, Tait, Tun\c{c}bilek, Yang, and Young \cite{desilva}
studied a rainbow version of Ramsey multiplicity.  
Instead of looking at 
the minimum number of monochromatic copies of $H$
over all $k$-colorings $c: E(K_n) \rightarrow [k]$, 
De Silva et.\ al\ look at the maximum number of 
rainbow copies of $H$.  One must consider rainbow copies of $H$
since giving every edge of $K_n$ the same color clearly maximizes
the number of monochromatic copies.  
Define $rb_k (H ; n)$ to be 
the maximum number of rainbow copies of $H$ 
over all $k$-colorings $c: E(K_n) \rightarrow [k]$.  This parameter 
is called the \emph{anti-Ramsey multiplicity} of $H$, and  
\cite{desilva} investigates the behavior of this function for different graphs $H$.  

In this paper, we consider an arithmetic analog of 
anti-Ramsey multiplicity thereby combining the problems raised in 
\cite{desilva} with the arithmetic setting of \cite{saad}.  
We will focus entirely on the Sidon equation $X+Y  = Z + T$.
The Sidon equation measures the additive energy of a set.   
The additive 
energy of a set $A \subset \Gamma$ 
is the number of four tuples $(a,b,c,d) \in A^4$ such that 
$a + b = c + d$. 
Typically it is written as 
\[
E(A) = | \{ ( a,b,c,d ) \in A^4 : a + b = c + d \} |.
\]
This fundamental 
parameter measures the additive structure of $A$, and for more on additive energy, see Tao and Vu \cite{tv}.
Additive energy is perhaps one of the reasons why the Sidon equation is used as a first example 
in \cite{saad}.  We would also like to remark that rainbow solutions to the Sidon equation were studied by Fox,     
Mahdian, and Radoi\v ci\'c \cite{fox}.  They proved that in every 4-coloring of $[n]$ 
where the smallest color classes has size at least $\frac{n+1}{6}$, there is at least one rainbow 
solution to the Sidon equation.   This result is also 
discussed in \cite{jnr} which surveys several problems 
on conditions ensuring a rainbow solution to an equation.

Since we are interested in the maximum number of rainbow solutions 
to $X+Y = Z+ T$,  
we must take a moment to carefully describe how solutions are counted.
First, since we are only counting rainbow solutions, we only 
care about solutions to $X+Y = Z + T$ in which all of the terms are distinct.  
Additionally, we want to count solutions that can be obtained 
by interchanging values on the same side of the equation 
as being the same.  
With this in mind, 
we define a set of four distinct integers 
$\{x_1 , x_2 , x_3 , x_4 \} \in \binom{ [n] }{4}$ 
a \emph{Sidon 4-set} if these integers form a solution to the 
Sidon equation $X+Y = Z + T$.  
Given a 
Sidon 4-set $\{x_1 , x_2 , x_3 , x_4 \}$, we can determine exactly which 
pairs appear on each side of $X+ Y = Z + T$.  
Without loss of generality, we may assume 
that $x_1$ is the largest among the $x_i$'s and $x_4$ is the smallest.
It follows that  $x_1 + x_4 = x_2 + x_3$
and again without loss of generality, we may assume $x_2 > x_3$ so $x_1 > x_2 > x_3 > x_4$. 
In short, given a Sidon 4-set $\{x_1,  x_2 , x_3 , x_4 \}$, the two 
extreme values appear on one side of $X+Y = Z + T$, and the two middle values appear 
on the other side.  

Now we are ready to define the Ramsey function that is the focus of this work.  
Let $n \geq k \geq 4$ be integers.  We define 
\[
AR_{ X+Y = Z + T }^k(n)
\]
to be the maximum number of rainbow 
Sidon 4-sets over  
all colorings $c: [n] \rightarrow [k]$.  
It can be shown that the total number of Sidon 4-sets in $[n]$ is exactly 
\[
\frac{n^3}{12} - \frac{3n^2}{8} + \frac{5n}{12} - \theta
\]
where $\theta = 0$ if $n$ is even, and $\theta = \frac{1}{8}$ if $n$ is odd.  
This immediately implies the upper bound
\[
AR_{X+Y = Z +T}^k (n) \leq \frac{n^3}{12}  - \frac{3n^2}{8} + \frac{5n}{12}
\]
for $n \geq k \geq 4$.  A First Moment Method argument gives a lower bound of 
\[
\left( \frac{1}{12} - \frac{1}{2k} + O \left( \frac{1}{k^2} \right) \right) n^3 - O_k (n^2) \leq AR_{X + Y = Z + T }^k (n) .
\]
Our first theorem improves both of these bounds.

\begin{theorem}\label{main theorem for k colors}
For integers $n \geq k \geq 4$, 
\[
\left( \frac{1}{12} - \frac{1}{3k} + \frac{\theta }{k^2} \right) n^3 - O_k (n^2) 
\leq 
AR_{X+Y = Z+T}^k (n) \leq \left( \frac{1}{12} - \frac{1}{24k} \right)n^3 + O_k(n^2)
\]
where $\theta = \frac{1}{3}$ if $k$ is even, and $\theta = \frac{1}{4}$ if $k$ is odd.  
\end{theorem}

When $k = 4$, we can improve the upper bound using a different argument.

\begin{theorem}\label{main result for 4 colors}
For $n \geq 4$, 
\[
\frac{2n^3}{96} - O(n^2) 
\leq 
AR_{X+Y = Z+T}^4 (n) \leq \frac{3n^3}{96} + O(n^2).
\]
\end{theorem}

The lower bound in Theorem \ref{main result for 4 colors} is a consequence 
of Theorem \ref{main theorem for k colors}.
Finding an asymptotic formula for $AR_{X + Y = Z + T}^k (n)$ is an open problem.

The rest of this paper is organized as follows.  In Sections \ref{section 2} and \ref{section 3}, we prove the upper bounds of 
Theorem \ref{main theorem for k colors} and Theorem \ref{main result for 4 colors}, respectively.  
The lower bound is proved in 
Section \ref{section 4}.  Some concluding remarks and further discussion is given in Section \ref{conclusion}.


\section{An Upper bound for $k$ colors}\label{section 2}

Key to our upper bound for $k > 4$ is the following lemma.  It gives a lower bound for the number 
of Sidon 4-sets that contain a fixed pair.  It will be applied to 
pairs that are monochromatic under a given coloring $c$.

\begin{lemma}\label{simpler pairs counting}
Let $n $ be a positive integer and let $\{ b < a \} \in \binom{ [n] }{2}$.  Define $f_n( \{ b < a \})$  to be the number of Sidon 4-sets $\{x_1 , x_2 , x_3 , x_4 \} \in \binom{ [n] }{4}$ with 
$\{a ,b  \} \subset \{x_1 , x_2 , x_3 , x_4 \}$. Then $f_n$ satisfies 
\[
f_n ( \{ b < a \} ) \geq \frac{n}{2} - 4.
\]
\end{lemma}

\begin{proof}
We will consider two possibilities depending on the positioning of $a$ and $b$ within the equation $x_1 + x_4 = x_2 + x_3$.

\bigskip
\noindent
\textit{Claim 1:} If $a + b = x_i + x_j$, then the number of $x_i$, $x_j \in [n]$ with $x_i < x_j$ that satisfy 
this equation is at least 
\[
\left\{
\begin{array}{ll}
\lfloor\frac{a+b-1}{2}\rfloor -1 & \mbox{if $a + b \leq n+1$,} \\
n- \lfloor\frac{a+b-1}{2}\rfloor -1 & \mbox{if $a+b > n+1$}.
\end{array}
\right.
\]
\medskip
\noindent
\textit{Proof of Claim 1:} First suppose $a + b \leq n + 1$.
Let $x_i = m$ and $x_j = a + b - m$ where $1 \leq m \leq \lfloor \frac{ a + b - 1}{2} \rfloor$. 
Since $a+b \leq n+1$, the integers $x_i$ and $x_j$ are in $[n]$ for all $m$ in the specified range. 
Since we must exclude $x_i = b$, $x_j = a$ ($a$, $b$, $x_i$, and $x_j$ must all be distinct to be a Sidon 4-set), 
we obtain that the total amount of possible values for $x_i, x_j$ is at least $\lfloor \frac{a+b-1}{2} \rfloor -1$.

\medskip

Now suppose $a+b > n + 1$. Let $x_i = a + b - n +m$ and $x_j = n - m$ where $0 \leq m \leq n - \lfloor \frac{a+b-1}{2} \rfloor$. 
Since $a+b > n + 1$, we have $n - m \leq x_i < x_j \leq n$ for all $m$ in the specified range.  As before, 
the solution $x_i  = b$ and $x_j = a$ must be excluded.  Here we obtain that the total amount of possible values for 
$x_i$ and $x_j$ is at least 
\[
n - \left\lceil \frac{ a + b - 1 }{2} \right\rceil \geq n - \left\lfloor \frac{ a + b - 1}{2} \right\rfloor -1.
\]  

\bigskip
\noindent
\textit{Claim 2:} If $a +x_i = b +x_j$, then the number of $x_i , x_j \in [n]$ with $x_i < x_j$ that satisfy this equation is 
at least $n - (a-b) - 3$.

\medskip
\noindent
\textit{Proof of Claim 2:} Note that $a +x_i = b +x_j$ implies $a -b =  x_j - x_i$. Since $b <a$, we have that $a-b > 0$. Let $x_i = m$ and $x_j = a - b + m$.  The range of $m$ for which we have a valid solution is $1 \leq m \leq n - (a-b)$. However, 
we also require that $\{a , b \} \cap \{ x_i , x_j \} = \emptyset$ and so the solutions 
$(x_i , x_j) = ( 2b - a , b)$, $(x_i , x_j) = (b ,a)$, and $(x_i , x_j)  = ( a , 2a - b )$ must all be excluded.  
Thus, we obtain the number $x_i$, $x_j$ that satisfy the equation $a + x_i = b + x_j$ and all other constraints 
is at least $n -(a-b) -3$. 

\medskip

These two possibilities ($a+b = x_i + x_j$ and $a+x_i = b + x_j$) are disjoint and cover all possible positions for $a$ and $b$. 
A lower bound on the number of Sidon 4-sets $\{x_1 , x_2 , x_3 , x_4 \} \in \binom{ [n] }{4}$ with 
$\{a ,b  \} \subset \{x_1 , x_2 , x_3 , x_4 \}$ is obtained by combining these two cases. 
So we have that if $a+b \leq n+1$, then the amount of Sidon 4-sets that contain $a$ and $b$ is at least 
\[
\left\lfloor\frac{a+b-1}{2} \right\rfloor -1 + n-(a-b) -3 \geq n - \frac{a}{2} + \frac{3b}{2} - \frac{11}{2} \geq  \frac{n}{2} - 4.
\]
If  $a+b > n+1$, then the amount of Sidon 4-sets that contain $a$ and $b$ is at least
\[
n- \left\lfloor\frac{a+b-1}{2} \right\rfloor -1 + n-(a-b) -3 \geq 2n - \frac{3a}{2} + \frac{b}{2} - \frac{9}{2} \geq  \frac{n}{2} - 4.
\]
\end{proof}

\bigskip

\begin{theorem}\label{general upper bound for all k}
For integers $n \geq k \geq 4$, 
\[
AR_{X+Y =  Z +T}^k (n) \leq \left( \frac{1}{12}  - \frac{1}{24k } \right)n^3 + O_k(n^2).
\]
\end{theorem}
\begin{proof}
Let $c: [n] \rightarrow [k]$ be a $k$-coloring of $[n]$.
Let $X_i$ be the integers assigned color $i$ by $c$.  
Let $\mathcal{M} \subset \binom{ [n] }{2}$ be the set of all pairs $\{ b < a \}$ which are 
monochromatic under $c$, i.e., $c(a) = c(b)$.  
Then  
\begin{equation}\label{100}
| \mathcal{M} | = \sum_{i = 1}^k \binom{ |X_i  |}{2} =  \frac{1}{2} \sum_{i = 1}^k | X_i |^2 -  \frac{1}{2} \sum_{i = 1}^k |X_i | 
 \geq  
\frac{1}{2} k \left( \frac{n}{k} \right)^2 - \frac{n}{2} = \frac{n^2}{2k} - \frac{n}{2}.
\end{equation}
Let $f_n ( \{ b < a \} )$ be the number of Sidon 4-sets in $[n]$ that contain $\{b < a \}$.  By Lemma 
\ref{simpler pairs counting}, 
\begin{equation}\label{200}
f_n ( \{ b < a \} ) \geq \frac{n}{2} - 4.
\end{equation}
The sum $\sum_{ \{ b < a \} \in \mathcal{M} } 
 f_n( \{ b < a \} )$
  counts the number of Sidon 4-sets that contain at least one monochromatic pair.
 A given Sidon 4-set is counted at most six times by this sum since there are $\binom{4}{2}$ ways to 
 choose a pair from a Sidon 4-set.  In fact, the only Sidon 4-sets that will be counted six times 
 in this sum are those which are monochromatic under $c$.  All others will be counted at most three times.  Regardless,
 we have that the number of Sidon 4-sets that are not rainbow under $c$ is at least 
 \[
 \frac{1}{6} \sum_{ \{ b < a \} \in \mathcal{M} } 
 f_n( \{ b < a \} )
 \geq \frac{1}{6} \sum_{ \{ b < a \} \in \mathcal{M} }  \left( \frac{n}{2} - 4 \right)
  \geq \frac{1}{6} \left( \frac{n^2}{2k} - \frac{n}{2} \right) \left( \frac{n}{2} - 4 \right) 
  = \frac{n^3}{24k} - O_k(n^2)
 \]
where we have used both (\ref{100}) and (\ref{200}).  
\end{proof}


\section{An Upper Bound for four colors}\label{section 3}

For $k =4$, the upper bound of Theorem \ref{general upper bound for all k} gives 
\[
AR_{X+Y=Z+T}^4 (n) \leq \left( \frac{1}{12} - \frac{1}{96} \right) n^3 + O(n^2).
\]
In the special case that $k=4$, we can obtain a better upper bound with a different 
argument based on additive energy.

Let $A_1 , A_2 , \dots , A_t$ be finite sets of integers and define 
\[
E_t ( A_1 , A_2 , \dots , A_t ) = 
| \{ (a_1 , a_2 , \dots , a_t ) \in A_1 \times A_2 \times \dots \times A_t : 
a_1 + a_2 + \dots + a_t = 0 \} |.
\]
For integers $n \leq m$, write $[n,m]$ for the interval 
\[
\{n , n + 1 , n + 2 , \dots , m \}.
\]
For a finite set $J \subset \mathbb{Z}$ with $j$ elements, let 
\[
I(J)  = [  - \lceil j/2 \rceil , \lceil j/2 \rceil ].
\] 
Note that $I(J)$ depends only on the cardinality of $J$.

A key ingredient in the proof of our upper bound is the following result of Lev \cite{lev}.

\begin{theorem}[Lev \cite{lev}]\label{lev theorem}
Let $t \geq 2$ be an integer.  
For any finite sets $A_1 , A_2 , \dots , A_t \subset \mathbb{Z}$,  
\[
E_t ( A_1 , A_2 , \dots , A_t ) \leq E_t ( I(A_1) , I( A_2) , \dots ,  I (A_t ) ).
\]
\end{theorem}

The main idea is to apply Theorem \ref{lev theorem} with $t = 4$, where 
$A_1, A_2 , A_3,A_4$ are color classes of a coloring $c:[n] \rightarrow [4]$.   
Before using Theorem \ref{lev theorem}, we need a few lemmas.  

For finite sets $A , B \subset \mathbb{Z}$ and an integer $m$, let 
\[
r_{A + B} (m ) = 
| \{ ( a ,b) \in A \times B : a + b = m \} |.
\]
 
 \begin{lemma}\label{two sets}
 Let $1 \leq \alpha \leq \beta$ be integers.  If $A = [ - \alpha , \alpha ]$ 
 and $B = [- \beta , \beta ]$, then  
 \[
 r_{A + B } (m) = 
 \left\{
 \begin{array}{ll}
 2 \alpha + 1 & \mbox{if $|m| \leq \beta - \alpha$}, \\
 \beta + \alpha + 1 - |m| & \mbox{if $\beta - \alpha \leq |m| \leq \alpha + \beta$}, \\
 0 & \mbox{otherwise.}
 \end{array}
 \right.
 \]
In particular, $r_{A+B} (m ) \leq \beta + \alpha + 1 - | m |$ whenever $|m| \leq \alpha + \beta$.  
\end{lemma}
\begin{proof}
For any $m$ with $|m| \leq \beta - \alpha$, we can write $m = j + ( m  -  j )$ where $j \in [ - \alpha , \alpha ]$. 
The term $m-j$ is in $B$ since if $|m| \leq \beta - \alpha$ and $|j| \leq \alpha$, then 
\[
| m - j | \leq |m| + |j| \leq \beta - \alpha + \alpha = \beta .
\]
This shows that $r_{A + B }(m) = 2 \alpha + 1$ whenever $|m| \leq \beta - \alpha$.  

Now suppose $  \beta  - \alpha \leq m \leq \alpha + \beta$, say $m = \beta - \alpha + l$ 
for some $l \in \{ 0 , 1, \dots , 2 \alpha \}$.  
Then
\begin{equation}\label{expression for k}
m = \beta  - \alpha + l = ( - \alpha + l + t) + ( \beta - t )
\end{equation}
for $t \in \{0,1, \dots , 2 \alpha - l \}$.  We now check that for each such $t$, 
we have $- \alpha + l + t \in A$ and $\beta - t \in B$.  
Since $\alpha \leq \beta$ and $0 \leq l \leq 2 \alpha$, 
\[
- \beta \leq \beta - 2 \alpha \leq \beta - 2 \alpha + l = \beta - ( 2 \alpha - l )  \leq \beta - t \leq \beta
\]
so $| \beta - t | \leq \beta$ hence $\beta - t \in B$.
Similarly,
\[
- \alpha \leq - \alpha + l + t \leq - \alpha + l + 2 \alpha - l = \alpha
\]
so $- \alpha + l  + t \in A$.  
Therefore, in (\ref{expression for k}), the term $- \alpha + l + t$ belongs to $A$ and 
$\beta - t$ belongs to $B$.  Furthermore, this is all of the ways to write $m$ as a 
sum of an integer in $A$ and an integer in $B$.  We conclude that 
for $\beta - \alpha \leq m \leq \alpha + \beta$, $r_{ A + B}(m) = \beta + \alpha + 1 - m$.
The proof is completed by noting that if $m > \alpha + \beta$, then $r_{A + B}(m) = 0$, and 
$A$ and $B$ are symmetric about 0 so that $r_{A +B}(m) = r_{A+B}(-m)$.

As for the assertion that $r_{A + B} (m ) \leq \beta + \alpha  + 1 - |m|$
for $|m| \leq \alpha + \beta$, it is enough to check 
that $\beta + \alpha + 1 - |m| \geq 2 \alpha + 1$ for $|m| \leq \beta - \alpha$.  An easy  
computation shows that these two inequalities are equivalent.    
\end{proof}

\begin{lemma}\label{one set}
If $\alpha$ is a positive integer and $J  = [ - \alpha , \alpha ]$, 
then 
 \[
 r_{J + J} (m) = 
 \left\{
 \begin{array}{ll}
 2 \alpha + 1 - |m| & \mbox{if $|m| \leq 2 \alpha$}, \\
 0 & \mbox{otherwise.}
 \end{array}
 \right.
 \]
\end{lemma}
\begin{proof}
Apply Lemma \ref{two sets} with $\alpha = \beta$.    
\end{proof}

\begin{lemma}\label{four sets}
Let $\alpha_1 , \alpha_2 , \alpha_3 , \alpha_4$ be positive integers such that 
$\alpha := \alpha_1 + \alpha_2 + \alpha_3 + \alpha_4 $ is divisible by 4.  
If $A_i = [ - \alpha_i , \alpha_i ]$ and $J = [ - \alpha / 4 , \alpha / 4]$, then for any integer $m$ with 
$|m| \leq \alpha /2$,  
\[
r_{A_1 + A_2} (m) + r_{A_3 + A_4} (m) \leq 2 r_{J+J} (m).
\]
\end{lemma}  
\begin{proof}
Let $m$ be an integer with $|m| \leq \frac{\alpha}{2}$.  
By Lemma \ref{two sets}, 
\begin{eqnarray*}
r_{A_1 +A_2}(m)  + r_{A_3 + A_4}(m) & \leq & 
\alpha_1 + \alpha_2 + 1 - |m| + \alpha_3 + \alpha_4 + 1 - |m|  \\
& = & 2 ( \alpha /2 + 1 - |m| ) = 2 r_{J+J} (m).
\end{eqnarray*}
For the last equality, we have used Lemma \ref{one set} with $J = [ - \alpha /4 , \alpha / 4]$.   
\end{proof}

\begin{lemma}\label{super energy bound}
Let $\alpha_1 , \alpha_2 , \alpha_3 , \alpha_4$ be positive integers such that 
$\alpha := \alpha_1 + \alpha_2 + \alpha_3 + \alpha_4 $ is divisible by 4.  
If $A_i = [ - \alpha_i , \alpha_i ]$ and $J = [ - \alpha / 4 , \alpha / 4]$, then 
\[
\sum_{ m \in \mathbb{Z} } r_{A_1 + A_2}(m) r_{A_3 + A_4}(m) 
\leq 
\sum_{m= - \frac{\alpha}{2}}^{ \frac{\alpha}{2}}  r_{J + J } (m)^2.
\]
\end{lemma}
\begin{proof}
First we show that if $|m| > \frac{\alpha}{2}$, then the product 
\[
r_{A_1 + A_2} (m) r_{A_3 + A_4 } (m)
\]
must be 0.  If $r_{A_1 + A_2}(m) \neq 0$ and 
$r_{A_3 + A_4 }(m) \neq 0$, then by Lemma \ref{two sets}, 
\begin{center}
$|m| \leq \alpha_1 + \alpha_2$~~ and~~ 
$|m| \leq \alpha_3 + \alpha_4$.
\end{center}
Adding the two inequalities together gives $|m| \leq \frac{\alpha_1 + \alpha_2 + \alpha_3 + \alpha_4}{2} $ and 
so $|m| \leq \frac{ \alpha }{2}$.   
Thus, 
\[
\sum_{ m \in \mathbb{Z} } r_{A_1 + A_2}(m) r_{A_3 + A_4}(m)  = 
\sum_{m = - \frac{\alpha}{2} }^{ \frac{\alpha}{2} } r_{A_1 + A_2}(m) r_{A_3 + A_4}(m) .
\]
By Lemma \ref{four sets}, for any $m$ with $|m| \leq \frac{\alpha}{2}$, 
we have $r_{A_1 + A_2}(m)+  r_{A_3 + A_4}(m) 
\leq 2  r_{J + J } (m)$.
Thus, the product $r_{A_1+A_2}(m) r_{A_3 + A_4}(m)$ is at most $r_{J + J}(m)^2$.  Since 
this holds for all $m$ with $|m| \leq \frac{\alpha}{2}$, 
\[
\sum_{ m \in \mathbb{Z} } r_{A_1 + A_2}(m) r_{A_3 + A_4}(m)  =  
\sum_{m = - \frac{\alpha}{2} }^{ \frac{\alpha}{2} } r_{A_1 + A_2}(m) r_{A_3 + A_4}(m)  
 \leq  
\sum_{m= - \frac{\alpha}{2} }^{ \frac{\alpha}{2} }  r_{J + J } (m)^2
\]
which completes the proof of the lemma.
\end{proof}

\begin{lemma}\label{energy count}
If $\alpha$ is a positive integer and $J = [ - \alpha  , \alpha ]$, then  
\[
E_4 ( J,J,J,J ) = \frac{16 \alpha^3}{3} + 8 \alpha^2 + \frac{14 \alpha }{3} + 1.
\]
\end{lemma}
\begin{proof}
We must count the number of 4-tuples $(x_1 , x_2, x_3 , x_4)$ with $ - \alpha \leq x_i \leq \alpha$ 
and 
\[
x_1 + x_2 + x_3 + x_4 = 0.
\]
For an integer $m$ with $ 0 \leq |m| \leq 2 \alpha $, we have $r_{J+J} (m) = 2 \alpha + 1 - |m|$
by Lemma \ref{one set}.  
The number of 
4-tuples $(x_1 , x_2, x_3 , x_4) $ with $ - \alpha \leq x_i \leq \alpha$  and 
$x_1 + x_2 + x_3 + x_4 = 0$ is 
\begin{eqnarray*}
\sum_{m = - 2 \alpha }^{ 2 \alpha } r_{J+J} (m) r_{J + J} (-m) 
& = &  r_{J+J} (0)^2  + 2 \sum_{ m = 1}^{ 2 \alpha } r_{J+J} (m)^2  \\
&=&  \left( 2 \alpha + 1 \right)^2 + 2 \sum_{ m  =1}^{ 2 \alpha } \left( 2 \alpha + 1 - m \right)^2 
 \\
  &= &\frac{16 \alpha^3}{3} + 8 \alpha^2 + \frac{14 \alpha }{3} + 1.
\end{eqnarray*}
\end{proof}

\begin{theorem}\label{newest ub}
The function $AR_{X + Y = Z + T}^4 (n)$ satisfies 
\[
AR_{X+Y = Z+T}^4(n) \leq \frac{3n^3}{96} + O(n^2).
\]
\end{theorem}
\begin{proof}
First we assume that $n$ is divisible by 8.  An easy monotonicity argument will complete the proof for all $n$.  

Suppose $c: [n] \rightarrow \{1,2,3,4 \}$ is a 4-coloring of $[n]$. Let $X_i$ be the integers 
assigned color $i$ by $c$ and $|X_i| = c_i n$.  
The number of rainbow solutions to $X + Y = Z + T$ is exactly 
\[
N(c) := E_4 (X_1 , X_2 , - X_3 , -X_4 ) + E_4 ( X_1 , X_3 , -X_2 , -X_4 ) + E_4 ( X_1 , X_4 , - X_2 , - X_3).
\]
By Theorem \ref{lev theorem},  
\begin{eqnarray*}
N(c) & \leq & E_4 (I(X_1) , I(X_2) , I(- X_3) , I(-X_4) ) + E_4 ( I(X_1) , I(X_3) , I(-X_2) , I(-X_4) ) \\
&+& E_4 ( I(X_1) , I( X_4 ) , I( -X_2 ) , I( -X_3 )).
\end{eqnarray*}
We will show that each of the terms on the right hand side is at most $\frac{n^3}{96} + O(n^2)$.  

For $1 \leq i \leq 4$, 
\[
I( \pm X_i) = [  - \lceil c_i  n/ 2 \rceil , \lceil c_i n / 2 \rceil ].
\]
We also have that 
$c_1  + c_2 + c_3 + c_4 \leq 1$. 
Assume that each $\frac{c_i n }{2}$ is an integer.
Let $A_1 = I(X_1)$, $A_2 = I(X_2)$, $A_3 = I(-X_3)$, $A_4 = I( - X_4)$, and 
$J= [-n/8 , n/8]$.        
By Lemmas \ref{super energy bound} and \ref{energy count}
\begin{eqnarray*}
E_4 (A_1 , A_2 , A_3 , A_4) & = & \sum_{m \in \mathbb{Z} } r_{A_1 + A_2 }(m) r_{A_3 + A_4}(-m) 
= \sum_{m \in \mathbb{Z} } r_{A_1 + A_2 }(m) r_{A_3 + A_4}(m)  \\
& \leq & \sum_{m = -n/4 }^{ n/4} r_{J + J}(m)^2 = 
E_4 (J,J,J,J) = \frac{n^3}{96}  + O(n^2).
\end{eqnarray*}
We apply this same estimate to 
$E_4 ( X_1 , X_3 , -X_2 , -X_4 )$ and $E_4 ( X_1 , X_4 , - X_2 , - X_3)$
to obtain
\[
N(c) \leq \frac{3n^3}{96} + O(n^2).
\]
If the $\frac{c_i n}{2}$ are not integers, we can still apply the above argument but now $J$ must be 
replaced with $J = [ - n/8 -1 , n/8 + 1]$.  Nevertheless, we still have 
$E_4 (J,J,J,J) \leq \frac{n^3}{96} + O(n^2)$ as $E_4 (J,J,J,J)$ increases by $O(n^2)$ when the 
interval $J$ is increases from $[-n/8 , n/8]$ to $[-n/8 - 1 , n/8 + 1]$.  

If $n$ is not divisible by $8$, then let 
$l$ be the smallest integer for which $n + l$ is divisible by 8 (so $1 \leq l \leq 7$). 
By monotonicity,
\[
AR_{X  + Y = Z + T}^4(n ) \leq AR_{X  + Y = Z + T}^4(n + l) \leq \frac{3( n+l)^3 }{96} + O ( (n + l )^2 ) 
 = \frac{3n^3}{96} + O(n^2).
 \]
\end{proof}


\section{A Lower Bound for $k$ colors}\label{section 4}

In this section we prove a lower bound on $AR_{X+Y = S + T}^k (n)$ for $ k \geq 4$.  
We will need two lemmas before proving the lower bound.  In this section, 
we continue to write 
\[
r_{A + B}(m) = | \{ (a,b ) \in A \times B : a + b = m \} |.
\]

\begin{lemma}\label{another new lemma}
Let $1 \leq i < j \leq k$ be integers and let $n$ be a positive integer that is divisible by $k$.  
If $X_i = \{ m \in [n] : m \equiv i ( \textup{mod}~k) \}$ and 
$X_j = \{ m \in [n] : m \equiv j ( \textup{mod}~k) \}$,
then 
\[
r_{X_i + X_j }( i  + j + tk ) \geq
\left\{
\begin{array}{ll}
t +1 & \mbox{if $0 \leq t \leq \frac{n}{k} - 1$,} \\
\frac{2n}{k} - 1 - t  & \mbox{if $\frac{n}{k} \leq t \leq \frac{2n}{k} - 2$.}
\end{array}
\right.
\]
If $l \not\equiv i + j ( \textup{mod}~k)$, then $r_{X_i + X_j} (l ) = 0$. 
\end{lemma}
\begin{proof}
First note that since $k$ divides $n$, 
\begin{center}
$X_i = \{ i , i + k , i + 2k , \dots , i + n - k \}$ ~~and~~ 
$X_j = \{ j , j + k , j + 2k , \dots , j + n - k \}$.
\end{center}

If $l = a + b$ for some $a \in X_i$ and $b \in X_j$, then $l \equiv i + j ( \textup{mod}~k)$.  
Thus, $r_{X_i + X_j }(l) = 0$ whenever $l \not\equiv i + j ( \textup{mod}~k)$. 
This proves the last assertion of the lemma.  

Let $t$ be an integer with $0 \leq t \leq \frac{n}{k} - 1$.  
We claim that for each $\alpha \in \{0,1, \dots , t \}$, we get
\[
i + j + tk = (  i + \alpha k ) + ( j + ( t - \alpha ) k )
\]
where $i + \alpha k \in X_i$ and $j + (t - \alpha ) k \in X_j$.
The inequality 
\[
i \leq i + \alpha k \leq i + tk \leq i + \left( \frac{n}{k} -1 \right) k = i + n - k
\]
shows that $i + \alpha k \in X_i$ for each $\alpha \in \{0,1, \dots , t \}$.
Similarly, 
\[
j \leq j + ( t - \alpha ) k \leq j + tk \leq j + \left( \frac{n}{k} -1 \right) k  = j + n - k 
\]
shows that $j + \alpha k \in X_j$ for each $\alpha \in \{0,1, \dots , t \}$.  
Consequently,
\[
r_{X_i + X_j } ( i + j + tk ) \geq t+1
\]
whenever 
$t \in \{0,1, \dots , \frac{n}{k} - 1 \}$. 

Now let $t$ be an integer with $\frac{n}{k} \leq t \leq \frac{2n}{k} - 2$.  
Write $t = \frac{2n}{k} - \beta$ where $2 \leq \beta \leq \frac{n}{k}$.  
For each $\alpha \in \{1,2 , \dots , \beta   - 1 \}$, we can write 
\[
i + j + tk = i + j + \left( \frac{2n}{k} - \beta \right) k = 
\left( i + \left( \frac{n}{k} - \alpha \right)k \right) + 
\left( j + \left( \frac{n}{k} - ( \beta - \alpha ) \right)k \right).
\]
We claim that $i + \left( \frac{n}{k} - \alpha \right) k \in X_i$ and 
$j + \left( \frac{n}{k} - ( \beta - \alpha ) \right)k \in X_j$.  
Now 
\[
i + k = 
i + n - \left( \frac{n}{k} -1 \right) k 
\leq 
i + n - ( \beta - 1 ) k \leq 
i + \left( \frac{n}{k} - \alpha \right) k \leq i + n - k
\]
where we have used the inequalities $\beta \leq \frac{n}{k}$, $\alpha \leq \beta -1$, and 
$\alpha \geq 1$.  
We conclude that for each $\alpha \in \{1,2, \dots , \beta - 1 \}$, the term 
$i + \left( \frac{n}{k} - \alpha \right)k$ is in $X_i$.  
Similarly, 
\begin{eqnarray*}
j + k & = & j + \left( \frac{n}{k} - \left( \frac{n}{k} - 1 \right) \right) k 
\leq 
j + \left( \frac{n}{k} - ( \beta - 1 ) \right) k 
\leq j + \left( \frac{n}{k} - ( \beta - \alpha ) \right)k \\
& \leq & j + \left( \frac{n}{k} - 1 \right) k = j + n - k 
\end{eqnarray*}
shows that $j + \left( \frac{n}{k} - ( \beta - \alpha ) \right)k$ is in $X_j$ for 
each $\alpha \in \{1,2, \dots , \beta  - 1 \}$. 
Therefore, 
\[
r_{X_i + X_j} ( i  + j + t k ) \geq \beta - 1.  
\]
Since 
$t = \frac{2n}{k} - \beta$, we have $\beta - 1 = \frac{2n}{k} - t - 1$ and this completes the proof of the lemma.
\end{proof}

\bigskip

For the next lemma we will count Sidon 4-sets in $\mathbb{Z}_k$.  
A \emph{Sidon 4-set in $\mathbb{Z}_k$} is a set of four distinct elements 
$\alpha$, $\beta$, $\gamma$, $\delta \in \mathbb{Z}_k$ such 
that $\alpha  + \beta \equiv \gamma + \delta ( \textup{mod}~k)$.
We will denote such a 4-set by $\{ \alpha + \beta \equiv \gamma + \delta \}$.
The reason we cannot simply write $\{ \alpha , \beta , \gamma , \delta \}$ is that 
in $\mathbb{Z}_k$, four distinct residues may lead to more than one solution to the Sidon equation.
For example, in $\mathbb{Z}_4$, 
\begin{center}
$1 + 2 \equiv 3 + 4 ( \textup{mod}~4)$ ~~and~~ $1 + 4 \equiv 2 + 3 ( \textup{mod}~4)$.
\end{center}
This does not occur in $\mathbb{Z}$ because of the ordering of the integers.  
Let $\mathcal{S}(k)$ be the collection of all Sidon 4-sets in $\mathbb{Z}_k$.
Finishing off the example of $k = 4$, it is easily seen that 
\begin{equation}\label{sidon in z4}
\mathcal{S}(4) = \{ \{ 1 + 2 \equiv 3 + 4 \} , \{ 1 + 4 \equiv 2 + 3 \} \}
\end{equation}
and so $| \mathcal{S}(4) | = 2$.  

We are now ready to state and prove the next lemma.  

\begin{lemma}\label{zk lemma}
Let $k \geq 4$ be an integer.  If $\mathcal{S}(k)$ is the family of all 
Sidon 4-sets in $\mathbb{Z}_k$, then 
\[
| \mathcal{S} (k) | = \frac{ k^3}{8}  - \frac{k^2}{2} + \theta k
\]
where $\theta = \frac{1}{2}$ if $k$ is even, and $\theta = \frac{3}{8}$ if $k$ is odd.  
\end{lemma}
\begin{proof}
For this lemma, we will write $a \equiv b$ for $a \equiv b ( \textup{mod}~k)$.  

Let us first assume that $k$ is even.  Where this will come into play is that when $k$ is even,
the congruence $2X \equiv b$  will have exactly two solutions when $b$ is even, and no solutions 
when $b$ is odd.  
First we choose a pair $\{x_1 , x_2 \} \in \binom{ \mathbb{Z}_k } {2}$.  This can be done in $\binom{k}{2}$ ways
and this pair will be one side of the equation $X + Y \equiv Z + T$.  
Our counting from this point forward depends on if $x_1 + x_2$ is even or odd when viewed as an integer.

\smallskip
\noindent
\textit{Case 1: $x_1 + x_2$ is even}

\smallskip
If $x_1 + x_2$ is even, then the congruence 
$2X \equiv x_1 + x_2$ has exactly two solutions, say $y_1$ and $y_2$.
Note that no $y_i$ can be the same as an $x_i$ for if, say $y_1 \equiv x_1$, 
then from $y_1 + y_1 \equiv x_1  + x_2$ we get $x_2 \equiv y_1 \equiv x_1$ contradicting 
the way $x_1$ and $x_2$ have been chosen.  Therefore, in the case that $x_1 + x_2$ is even, 
there are $k-4$ choices for $x_3$ for which the unique $x_4$ satisfying 
\[
x_1 + x_2 \equiv x_3 + x_4
\]
will have the property that all of $x_1$, $x_2$, $x_3$, and $x_4$ are distinct.  We conclude 
that 
\[
\{ x_1 + x_2 \equiv x_3 + x_4 \}
\]
is indeed a Sidon 4-set.  This Sidon 4-set 
is counted exactly four times in this way: we could have chosen $x_3$ or $x_4$ 
after having chosen the pair $\{x_1 , x_2 \}$, and we could have also started by choosing 
the pair $\{x_3 , x_4 \}$ instead.  
When $k$ is even, the number of pairs $\{x_1 , x_2 \}$ for which 
$x_1 + x_2$ is even is exactly $\sum_{t = 1}^{ \frac{k}{2} - 1 } 2t = \frac{k^2 }{4} - \frac{k}{2}$ (this 
can be seen by looking at the diagonals in a Cayley table for $\mathbb{Z}_k$).  
Altogether, we have a count of 
\[
\frac{ ( \frac{k^2}{4} - \frac{k}{2} )  ( k - 4) }{4}
\]
Sidon 4-sets $\{x_1 + x_2 \equiv x_3 + x_4 \}$ where $x_1 + x_2$ is even.

\smallskip
\noindent
\textit{Case 2: $x_1 + x_2$ is odd}

\smallskip
If $x_1 + x_2$ is odd, then $2X \equiv x_1 + x_2$ has no solution since $\textup{gcd}(2,k)$ 
does not divide $x_1+ x_2$.  Now there will be $k-2$ choices for $x_3$ and the unique $x_4$ satisfying 
$x_1 + x_2 \equiv x_3 + x_4$ will have the property that 
$\{x_1 , x_2 , x_3 , x_4 \}$ is a 4-set.  
There are $\sum_{t = 1}^{\frac{k}{2} } (2t - 1)  = \frac{k^2}{4}$
pairs $\{x_1 , x_2 \}$ for which $x_1 + x_2$ is odd.   
This gives a count of 
\[
\frac{ ( \frac{k^2}{4} )  ( k - 2) }{4}
\]
Sidon 4-sets $\{x_1 + x_2 \equiv x_3 + x_4 \}$ where $x_1 + x_2$ is odd.

\smallskip
Combining the two cases, there are exactly 
\[
\frac{ ( \frac{k^2}{4} - \frac{k}{2} )  ( k - 4) }{4} + \frac{ ( \frac{k^2}{4} )  ( k - 2) }{4} = \frac{k^3}{8} - \frac{k^2}{2} + \frac{k}{2}
\]
Sidon 4-sets in $\mathbb{Z}_k$ when $k$ is even.  

When $k$ is odd, a similar counting argument can be done.  The key difference is 
that for any pair $\{x_1 , x_2 \}$, the congruence $2X \equiv x_1 + x_2$ has exactly 
one solution since $\textup{gcd}(k,2) = 1$ always divides $x_1 + x_2$.  This unique solution must 
be avoided when choosing $x_3$ and so there will be $k-3$ choices for $x_3$.  The rest of 
the counting is similar to as before and we obtain 
\[
\frac{ \binom{k}{2} ( k - 3) }{4} = \frac{k^3}{8} - \frac{k^2}{2} + \frac{3k }{ 8 } 
\]
Sidon 4-sets in $\mathbb{Z}_k$ when $k$ is odd.  
\end{proof}

\begin{theorem}\label{k color lb}
Let $n \geq k \geq 4$ be integers and assume that $n$ is divisible by $k$.  If $\mathcal{S} (k) $ is the family of 
all Sidon 4-sets in $\mathbb{Z}_k$, then 
\[
AR_{X+Y = Z + T}^k(n) \geq 2 | \mathcal{S} ( k ) |
\left(  \frac{n^3}{3k^3} - O_k ( n^2) \right).
\]  
\end{theorem}
\begin{proof}
Let $n \geq k \geq 4$ be integers where $k$ divides $n$.  Define the coloring 
$c: [n] \rightarrow [k]$ by $c(i) = i ( \textup{mod}~k)$ where we use residues 
in the set $\{1,2, \dots ,k \}$.  The number of 
rainbow Sidon 4-sets under $c$ is 
\begin{equation}\label{big sum}
\sum_{l = 1}^{2n} 
\sum_{1 \leq i < j < s < t \leq k } 
\left(  r_{X_i + X_j} (l) r_{X_s + X_t} (l) +
 r_{X_i + X_s} (l) r_{X_j + X_t} (l) +
  r_{X_i + X_t} (l) r_{X_j + X_s} (l)
\right)
\end{equation}
where $X_i = \{ m \in [n] : m\equiv i ( \textup{mod}~k) \}$.  
To see this, observe that if $x_1 + x_2 = x_3 + x_4$ is a Sidon 4-set that is rainbow, then there are distinct colors 
$1 \leq i < j < s < t \leq k$ with 
\[
\{ c(x_1) , c(x_2) , c(x_3) , c(x_4) \} = \{ i , j , s, t \}.
\]
This rainbow Sidon 4-set is counted exactly once by the sum (\ref{big sum}) 
precisely when $l = x_1 + x_2$, and by only one of the terms in the sum
\begin{equation}\label{rainbow counting equation}
 r_{X_i + X_j} (l) r_{X_s + X_t} (l) +
 r_{X_i + X_s} (l) r_{X_j + X_t} (l) +
  r_{X_i + X_t} (l) r_{X_j + X_s} (l).
  \end{equation}
 The unique nonzero term depends on which two colors appear on the same side of the equation $x_1 + x_2 = x_3 + x_4$.  
For instance, if colors $i$ and $j$ appear on the same side, then the first term in (\ref{rainbow counting equation}) 
is the one that counts $\{x_1, x_2 , x_3 , x_4 \}$.  

Fix an $l \in [n]$ and four distinct colors $i,j,s,t$.   
By Lemma \ref{another new lemma}, 
the product 
\[
r_{X_i  + X_j } (l) r_{X_s + X_t } (l)
\]
is not zero 
if and only if $l \equiv i + j ( \textup{mod}~k)$ and 
$l \equiv s + t ( \textup{mod}~k)$.  This clearly implies 
$i + j \equiv s + t ( \textup{mod}~k)$ and so 
$\{ i + j \equiv s +t \}$ is Sidon 4-set in $\mathbb{Z}_k$.  
For $u \in \{0,1, \dots , k  - 1\}$, let 
\[
\mathcal{S}(k, u )
\]
be the Sidon 4-sets 
$\{ \alpha + \beta \equiv \gamma + \delta \} \in \mathcal{S}(k)$ for 
which $\alpha + \beta \equiv u ( \textup{mod}~k)$.  
The collection $\{ \mathcal{S} (k,u ) : 0 \leq u \leq k - 1 \}$ forms 
a partition of $\mathcal{S}(k)$.  
Since $r_{X_i + X_j}(l) \neq 0$ if only if $l \equiv i + j ( \textup{mod}~k)$, 
(\ref{big sum}) can be rewritten as 
\[
S:=
\sum_{l = 0}^{ \frac{2n}{k} -1  } 
\sum_{u = 1}^{k} 
\sum_{ \{ \alpha + \beta \equiv  \gamma + \delta \} \in \mathcal{S}(k,u) } 
r_{ X_{ \alpha } + X_{ \beta} } (u + kl ) r_{ X_{ \gamma } + X_{ \delta } }(u + k l).
\] 
In order to use Lemma \ref{another new lemma}, we split this sum into two sums $S_1$ and $S_2$ where $S \geq S_1 + S_2$.
Define  
\[
S_1:=
\sum_{l = 0}^{ \frac{n}{k} -1  } 
\sum_{u = 1}^{k} 
\sum_{ \{ \alpha + \beta \equiv  \gamma + \delta \} \in \mathcal{S}(k,u) } 
r_{ X_{ \alpha } + X_{ \beta} } (u + kl ) r_{ X_{ \gamma } + X_{ \delta } }(u + k l)
\] 
and
\[
S_2:=
\sum_{l = \frac{n}{k}}^{ \frac{2n}{k} -2  } 
\sum_{u = 1}^{k} 
\sum_{ \{ \alpha + \beta \equiv  \gamma + \delta \} \in \mathcal{S}(k,u) } 
r_{ X_{ \alpha } + X_{ \beta} } (u + kl ) r_{ X_{ \gamma } + X_{ \delta } }(u + k l).
\] 
By Lemma \ref{another new lemma}, 
\begin{equation}\label{s1 lower bound}
S_1  \geq  \sum_{l = 0}^{ \frac{n}{k} - 1 } \sum_{ u = 1}^{k} 
\sum_{ \{ \alpha + \beta \equiv \gamma + \delta \} \in \mathcal{S}(k,u) } ( l + 1)^2 
 \geq  \sum_{l = 0}^{ \frac{n}{k} - 1 } | \mathcal{S}(k) | ( l + 1)^2 
 =
 | \mathcal{S}(k) | \left( \frac{n^3}{3k^3} - O_k (n^2) \right).
\end{equation}
A similar application of Lemma \ref{another new lemma}
gives 
\begin{equation}\label{s2 lower bound}
S_2 \geq  | \mathcal{S}(k) | \left( \frac{n^3}{3k^3} - O_k (n^2) \right).
\end{equation}
Combining (\ref{s1 lower bound}) and (\ref{s2 lower bound}), we have 
\begin{equation}\label{s lower bound}
S \geq S_1 + S_2 \geq 2 | \mathcal{S}(k) | \left( \frac{n^3}{3k^3} - O_k (n^2) \right)
\end{equation}
which tells us that the number of rainbow Sidon 4-sets under the coloring $c$ is at least the
right hand side of (\ref{s lower bound}).  
\end{proof}

\begin{corollary}\label{corollary lb for k colors}
For integers $n \geq k \geq 4$, 
the function $AR_{X+Y = Z + T}^k (n)$ satisfies 
\[
AR_{X+Y = Z+T}^k (n) \geq \left( \frac{1}{12} - \frac{1}{3k} + \frac{ \theta }{k^2}  \right) n^3 - O_k(n^2) 
\]
where $\theta = \frac{1}{3}$ if $k$ is even, and $\theta = \frac{1}{4}$ if $k$ is odd.
\end{corollary}
\begin{proof}
 First assume that $n$ is divisible by $k$.  
 By Theorem \ref{k color lb} and Lemma \ref{zk lemma},
 \[
 AR_{X+Y=Z+T}^k (n) \geq  2 \left( \frac{k^3}{8} - \frac{k^2}{2} + \gamma k \right) \left( \frac{ n^3}{3  k^3 } - O_k(n^2) \right).
  \]
 where $\gamma = \frac{1}{2}$ if $k$ is even and $\gamma = \frac{3}{8}$ if $k$ is odd.  
 
 If $n$ is not divisible by $k$, then choose $r \in [k-1] $ so that $n - r$ is divisible by $k$.  
 We then have by monotonicity, 
 \[
  AR_{X+Y=Z+T}^k (n) \geq AR_{X+Y=Z+T}^k (n-r) \geq 2
  \left( \frac{k^3}{8} - \frac{k^2}{2} + \gamma k \right) \left( \frac{ (n-r)^3}{3  k^3 } - O_k(n^2) \right).
  \]
  The lower order term can be absorbed into the $O_k(n^2)$ error term so we get 
  \begin{eqnarray*}
 AR_{X+Y= Z+T}^k (n) & \geq &   2 \left( \frac{k^3}{8} - \frac{k^2}{2} + \gamma k \right) \left( \frac{ n^3}{3  k^3 } - O_k(n^2) \right) 
  \\
  & = &  \left( \frac{1}{12} - \frac{1}{3k} + \frac{ \theta }{k^2}  \right) n^3 - O_k(n^2) 
  \end{eqnarray*}
in either case.  Here $\theta = \frac{1}{3}$ if $k$ is even, and $\theta = \frac{1}{4}$ if $k$ is odd.

\end{proof}


\section{Concluding Remarks}\label{conclusion}

In this paper we studied the anti-Ramsey function $AR_{X+Y = Z + T}^k (n)$ which concerns colorings 
of $[n]$.  One could also consider colorings of $\mathbb{Z}_n$.  Write 
$AR_{X+Y \equiv  Z +T}^k ( \mathbb{Z}_n )$ for the maximum number of rainbow 
solutions to $X+Y \equiv Z+T ( \textup{mod}~n)$
over all $k$-colorings $c: \mathbb{Z}_n \rightarrow [k]$.  
 As in the case of $[n]$, we count solutions that
 only differ by ordering as the same.  This is discussed in 
detail prior to Lemma \ref{zk lemma}.  Now by Lemma \ref{zk lemma}, 
\[
AR^k_{X   +Y \equiv Z + T } ( \mathbb{Z}_n ) \leq \frac{n^3}{8} - \frac{n^2}{2} + \theta n
\]
where $\theta = \frac{1}{2}$ if $n$ is even, and $\theta = \frac{3}{8}$ if $n$ is odd.  
When $k = 4$, it is easy to improve this upper bound as follows.  Let $c: \mathbb{Z}_n \rightarrow [4]$ be a 
coloring of $\mathbb{Z}_n$ and let $X_i$ be the elements of $\mathbb{Z}_n$ assigned color $i$ by $c$.
The number of rainbow solutions to the Sidon equation $X + Y \equiv Z + T ( \textup{mod}~n)$ where colors 1 and 2 appear 
on the same side is at most 
\begin{equation}\label{product of three}
\min \{ |X_1 | |X_2 | |X_3 | , |X_1| |X_2| |X_4| , | X_1 | |X_3 | |X_4| , |X_2 | |X_3 | |X_4 | \}.
\end{equation}
Indeed, once we have chosen three values for the four variables $X$, $Y$, $Z$, and $T$, the last variable is uniquely 
determined.  Since $|X_1| + |X_2| + |X_3| +|X_4| = n$, (\ref{product of three}) is at most $\frac{n^3}{64}$.  
There are two other possible ways to obtain a rainbow solution to $X+Y \equiv Z + T ( \textup{mod}~n)$.  One is where 
colors 1 and 3 appear on the same side, and the other is where colors 1 and 4 appear on the same side.  This gives the 
upper bound 
\[
AR_{X + Y \equiv Z + T }^4 ( \mathbb{Z}_n ) \leq \frac{3n^3}{64}.
\]
As for a lower bound, a natural idea is to try the same coloring that is used to prove Theorem \ref{k color lb}.
It turns out that this is not more difficult if we consider arbitrary $k \geq 4$, nevertheless we restrict to $k =4$ for 
simplicity.  Define the coloring $c : \mathbb{Z}_n \rightarrow [4]$ by $c(i) = i ( \textup{mod}~4)$ 
where we use residues in $\{1,2,3,4 \}$ for the colors.  
If $n$ is not divisible by $4$, then 
this coloring may not be well defined!  A simple example is when $n = 5$ where $c(5) = 1$, and $c(10) = 2$, 
however, 5 and 10 are the same element of $\mathbb{Z}_5$.  An obvious way to fix this is to
fix equivalence class representatives, say $\mathbb{Z}_n = \{1,2, \dots , n \}$.
Unfortunately this does not solve the problem as 
we still require the arithmetic in $\mathbb{Z}_n$ when finding solutions to $X+Y \equiv Z + T ( \textup{mod}~n)$.
To proceed further, let us now assume that $n$ is divisible by 4 and so the coloring 
$c$ will be well defined and will not depend on how we represent the elements of $\mathbb{Z}_n$.    
It is now straightforward to adapt Lemma \ref{another new lemma} to the 
$\mathbb{Z}_n$ case.  For $1 \leq i < j \leq 4$, we would have 
\[
r_{X_i + X_j} ( i + j + 4t ) = \frac{n}{4}
\]
for all $t \in \{0,1, \dots , \frac{n}{4} - 1 \}$, and $r_{X_i + X_j } (l) = 0$ if $l \not\equiv i + j ( \textup{mod}~4)$.  
The proof of this follows along the same lines as the proof of Lemma \ref{another new lemma}, except now  
\[
i + j + 4 t  \equiv ( i + 4 \alpha  ) + ( j + 4 ( t - \alpha ) ) ( \textup{mod}~n)
\]
for all $\alpha \in \{1, \dots , \frac{n}{4} \}$.  One then obtains the lower bound
\begin{eqnarray*}
AR_{X+Y \equiv Z + T }^4 ( \mathbb{Z}_n ) & \geq & \sum_{l=0}^{ \frac{n}{4} - 1 } \sum_{u = 1}^4 
\sum_{ \{ \alpha + \beta \equiv \gamma + \delta \} \in \mathcal{S}(4 , u ) } 
r_{X_{ \alpha } + X_{ \beta} } ( u + 4l ) r_{X_{ \gamma} + X_{ \delta } } ( u + 4l ) \\
& = & \sum_{l = 0}^{ \frac{n}{4} - 1}  r_{X_1 + X_2}( 3 + 4l ) r_{X_3 + X_4 } ( 3 + 4l )
+r_{X_1 + X_4}( 1 + 4l ) r_{X_2 + X_3 } ( 1 + 4l )  \\
&  = & \sum_{l = 0}^{ \frac{n}{4} - 1} \left( \left( \frac{n}{4} \right)^2 + \left( \frac{n}{4} \right)^2 \right)
 =  \frac{n^3}{32}
\end{eqnarray*}
again, assuming $n$ is divisible by 4.  

When $k = 4$, determining an asymptotic formula for the number of rainbow solutions to the Sidon equation in $[n]$ or 
$\mathbb{Z}_n$ would certainly be interesting.  Additionally, improving the upper bound 
\[
AR_{X+Y = Z+ T}^k (n) \leq \left( \frac{1}{12} - \frac{1}{24k} \right)n^3 + O_k (n^2)
\]
seems possible.  Using the methods of this paper, one might be able to improve the $\frac{1}{24k}$ to 
$\frac{1}{12k}$, but we believe the lower bound is closer to 
the truth and so any significant improvement may require some 
new ideas.  


\section{Acknowledgment}

The authors would like to thank Michael Tait for suggesting the problem that is the focus of this paper and for some insightful 
comments.


\end{document}